\documentclass[11pt]{article}

\usepackage{amsmath, amssymb, amsthm, calc}  

\title{Intermediate Diophantine exponents \\ and \\ parametric geometry of numbers.
                              \thanks{ This research was  supported by
                              RFBR (grant $\textup N^{\circ}$ 09--01--00371a) and
                              by the grant of the President of Russian Federation
                              $\textup N^\circ$ MK--1226.2010.1.
                             }}
\author{Oleg\,N.\,German}
\date{}

\theoremstyle{definition}
\newtheorem{definition}{Definition}

\theoremstyle{remark}
\newtheorem{remark}{Remark}

\theoremstyle{plain}
\newtheorem{theorem}{Theorem}
\newtheorem{lemma}{Lemma}

\newtheorem{proposition}{Proposition}
\newtheorem{corollary}{Corollary}

\newtheorem{classic}{Theorem}
\newtheorem{classicprime}[classic]{Theorem}

\DeclareMathOperator{\spanned}{span}

\renewcommand{\vec}[1]{\mathbf{#1}}
\renewcommand{\geq}{\geqslant}
\renewcommand{\leq}{\leqslant}
\renewcommand{\phi}{\varphi}

\newcommand{\R}{\mathbb{R}}
\newcommand{\Z}{\mathbb{Z}}

\newcommand{\La}{\Lambda}

\newcommand{\bpsi}{\underline{\psi}}
\newcommand{\apsi}{\overline{\psi}}
\newcommand{\bPsi}{\underline{\Psi}}
\newcommand{\aPsi}{\overline{\Psi}}

\newcommand{\cL}{\mathcal{L}}

\newcommand{\cB}{\mathcal{B}}

\newcommand{\cJ}{\mathcal{J}}
\newcommand{\cP}{\mathcal{P}}
\newcommand{\cQ}{\mathcal{Q}}

\newcommand{\gT}{\mathfrak{T}}
\newcommand{\ga}{\mathfrak{a}}
\newcommand{\gb}{\mathfrak{b}}

\newcommand{\tr}[1]{{#1}^\intercal}

\begin{document}

  \maketitle

  \begin{abstract}
    In this paper we develop some of the ideas belonging to W.~Schmidt and L.~Summerer to define intermediate Diophantine exponents and split several transference inequalities into a chain of inequalities for intermediate exponents.
  \end{abstract}

  \section{Introduction}

  Given a matrix
  \[ \Theta=
     \begin{pmatrix}
       \theta_{11} & \cdots & \theta_{1m} \\
       \vdots & \ddots & \vdots \\
       \theta_{n1} & \cdots & \theta_{nm}
     \end{pmatrix},\qquad
     \theta_{ij}\in\R,\quad n+m\geq3, \]
  consider the system of linear equations
  \begin{equation} \label{eq:the_system}
    \Theta\vec x=\vec y
  \end{equation}
  with variables $\vec x\in\R^m$, $\vec y\in\R^n$.
  The classical measure of how well the space of solutions to this system can be approximated by integer points is defined as follows. Let $|\cdot|$ denote the sup-norm in the corresponding space.

  \begin{definition} \label{def:belpha_1}
    The supremum of the real numbers $\gamma$, such that there are arbitrarily large values of $t$ for which (resp. such that for every $t$ large enough) the system of inequalities
    \begin{equation} \label{eq:belpha_1_definition}
      |\vec x|\leq t,\qquad|\Theta\vec x-\vec y|\leq t^{-\gamma}
    \end{equation}
    has a nonzero solution in $(\vec x,\vec y)\in\Z^m\oplus\Z^n$, is called the \emph{regular} (resp. \emph{uniform}) \emph{Diophantine exponent} of $\Theta$ and is denoted by $\beta_1$ (resp. $\alpha_1$).
  \end{definition}

  This paper is a result of the attempt to generalize this concept to the case of the problem of approximating the space of solutions to \eqref{eq:the_system} by $p$-dimensional rational subspaces of $\R^{m+n}$. A large work in this direction was made by W.~Schmidt in \cite{schmidt_annals}. Later, in \cite{laurent_up_down}, \cite{bugeaud_laurent_up_down}, a corresponding definition was given by M.~Laurent and Y.~Bugeaud in the case when $m=1$. With their definition they were able to split the classical Khintchine transference principle into a chain of inequalities for intermediate exponents. However, the way we defined $\alpha_1$ and $\beta_1$ naturally proposes a generalization, which appears to be different from Laurent's:

  \begin{definition} \label{def:belpha_p}
    The supremum of the real numbers $\gamma$, such that there are arbitrarily large values of $t$ for which (resp. such that for every $t$ large enough) the system of inequalities
    \begin{equation} \label{eq:belpha_p_definition}
      |\vec x|\leq t,\qquad|\Theta\vec x-\vec y|\leq t^{-\gamma}
    \end{equation}
    has $p$ solutions $\vec z_i=(\vec x_i,\vec y_i)\in\Z^m\oplus\Z^n$, $i=1,\ldots,p$, linearly independent over $\Z$, is called the \emph{$p$-th regular (resp. uniform) Diophantine exponent of the first type} of $\Theta$ and is denoted by $\beta_p$ (resp. $\alpha_p$).
  \end{definition}

  In Section \ref{sec:laurexp} we propose a definition of intermediate exponents of the second type, which is consistent with Laurent's. In Sections \ref{sec:known}, \ref{sec:results} we formulate our main results for these quantities. Sections \ref{sec:schmex}, \ref{sec:multi_schmex} are devoted to the exponents naturally emerging in parametric geometry of numbers developed by W.~Schmidt and L.~Summerer in \cite{schmidt_summerer}. Those exponents are closely connected to the Diophantine exponents, and in Sections \ref{sec:ex_via_schmex}, \ref{sec:transposed} we describe this connection. It allows to reformulate our main results in terms of Schmidt--Summerer's exponents, which is accomplished in Section \ref{sec:results_schmidted}. Finally, in Section \ref{sec:proofs}, we use this point of view to prove Theorems given in Section \ref{sec:results}.

  It should be noticed that all our ``splitting'' results are obtained for the exponents of the second type. It is an interesting question whether anything of this kind can be done with the exponents of the first type.

  \section{Laurent's exponents and their generalization} \label{sec:laurexp}

  Set $d=m+n$. Let us denote by $\pmb\ell_1,\ldots,\pmb\ell_d$ the columns of the matrix
  \[ \begin{pmatrix}
       E_m & -\tr\Theta\\
       \Theta & E_n
     \end{pmatrix}, \]
  where $E_m$ and $E_n$ are the corresponding unity matrices and $\tr\Theta$ is the transpose of $\Theta$. Clearly, $\cL=\spanned_\R(\pmb\ell_1,\ldots,\pmb\ell_m)$ is the space of solutions to the system \eqref{eq:the_system}, and $\cL^\bot=\spanned_\R(\pmb\ell_{m+1},\ldots,\pmb\ell_d)$. Denote also by $\vec e_1,\ldots,\vec e_d$ the columns of the $d\times d$ unity matrix $E_d$.

  The following Definition is a slightly modified Laurent's one.

  \begin{definition} \label{def:ba_for_m_equal_to_1}
    Let $m=1$. The supremum of the real numbers $\gamma$, such that there are arbitrarily large values of $t$ for which (resp. such that for every $t$ large enough) the system of inequalities
    \begin{equation} \label{eq:ba_for_m_equal_to_1}
      |\vec Z|\leq t,\qquad|\pmb\ell_1\wedge\vec Z|\leq t^{-\gamma}
    \end{equation}
    has a nonzero solution in $\vec Z\in\wedge^p(\Z^d)$ is called the \emph{$p$-th regular (resp. uniform) Diophantine exponent of the second type} of $\Theta$ and is denoted by $\gb_p$ (resp. $\ga_p$).
  \end{definition}

  Here $\vec Z\in\wedge^p(\R^d)$, $\pmb\ell_1\wedge\vec Z\in\wedge^{p+1}(\R^d)$ and for each $q$ we consider $\wedge^q(\R^d)$ as a $\binom dq$-dimensional Euclidean space with the orthonormal basis consisting of the multivectors
  \[ \vec e_{i_1}\wedge\ldots\wedge\vec e_{i_q},\qquad 1\leq i_1<\ldots<i_q\leq d, \]
  and denote by $|\cdot|$ the sup-norm with respect to this basis.

  Laurent denoted the exponents $\gb_p$, $\ga_p$ as $\omega_{p-1}$, $\hat\omega_{p-1}$, respectively, and showed that for $p=1$ they coincide with $\beta_1$, $\alpha_1$. He also noticed that one does not have to require $\vec Z$ to be decomposable in Definition \ref{def:ba_for_m_equal_to_1}, which essentially simplifies working in $\wedge^p(\R^d)$.

  In order to generalize Definition \ref{def:ba_for_m_equal_to_1} let us set for each $\sigma=\{i_1,\ldots,i_k\}$, $1\leq i_1<\ldots<i_k\leq d$,
  \begin{equation} \label{eq:L_sigma}
    \vec L_\sigma=\pmb\ell_{i_1}\wedge\ldots\wedge\pmb\ell_{i_k},
  \end{equation}
  denote by $\cJ_k$ the set of all the $k$-element subsets of $\{1,\ldots,m\}$, $k=0,\ldots,m$, and set $\vec L_\varnothing=1$.

  Let us also set $k_0=\max(0,m-p)$.

  \begin{definition} \label{def:ba}
    The supremum of the real numbers $\gamma$, such that there are arbitrarily large values of $t$ for which (resp. such that for every $t$ large enough) the system of inequalities
    \begin{equation} \label{eq:ba}
      \max_{\sigma\in\cJ_k}|\vec L_\sigma\wedge\vec Z|\leq t^{1-(k-k_0)(1+\gamma)},\qquad k=0,\ldots,m,
    \end{equation}
    has a nonzero solution in $\vec Z\in\wedge^p(\Z^d)$ is called the \emph{$p$-th regular (resp. uniform) Diophantine exponent of the second type} of $\Theta$ and is denoted by $\gb_p$ (resp. $\ga_p$).
  \end{definition}

  We tended to make Definition \ref{def:ba} look as simple as possible. However, it will be more convenient to work with in the multilinear algebra setting after it is slightly reformulated. To give the desired reformulation let us set for each $\sigma=\{i_1,\ldots,i_k\}$, $1\leq i_1<\ldots<i_k\leq d$,
  \begin{equation} \label{eq:E_sigma}
    \vec E_\sigma=\vec e_{i_1}\wedge\ldots\wedge\vec e_{i_k},
  \end{equation}
  denote by $\cJ'_k$ the set of all the $k$-element subsets of $\{m+1,\ldots,d\}$, $k=0,\ldots,n$, and set $\vec E_\varnothing=1$.

  Set also $k_1=\min(m,d-p)$.

  \begin{proposition} \label{prop:ba_substitution}
    The inequalities \eqref{eq:ba} can be substituted by
    \begin{equation} \label{eq:ba_modified}
      \max_{\begin{subarray}{c} \sigma\in\cJ_k \\ \sigma'\in\cJ'_{d-p-k} \end{subarray}}
      |\vec L_\sigma\wedge\vec E_{\sigma'}\wedge\vec Z|\leq t^{1-(k-k_0)(1+\gamma)},\qquad k=k_0,\ldots,k_1.
    \end{equation}
  \end{proposition}

  \begin{proof}
    Since $\pmb\ell_1,\ldots,\pmb\ell_m,\vec e_{m+1},\ldots,\vec e_d$ form a basis of $\R^d$, for each $q=1,\dots,d$ the multivectors
    \[ \vec L_\rho\wedge\vec E_{\rho'},\qquad\rho\in\cJ_j,\ \rho'\in\cJ'_{q-j},\ \max(0,q-n)\leq j\leq\min(q,m), \]
    form a basis of $\wedge^q(\R^d)$. Let us denote by $|\cdot|_\Theta$ the sup-norm in each $\wedge^q(\R^d)$ with respect to such a basis. Since any two norms in a Euclidean space are equivalent, and since in Definition \ref{def:ba} we are concerned only about exponents, we can substitute \eqref{eq:ba} by
    \begin{equation} \label{eq:ba_Thetanized}
      \max_{\sigma\in\cJ_k}|\vec L_\sigma\wedge\vec Z|_\Theta\leq t^{1-(k-k_0)(1+\gamma)},\qquad k=0,\ldots,m,
    \end{equation}
    and \eqref{eq:ba_modified} by
    \begin{equation} \label{eq:ba_modified_Thetanized}
      \max_{\begin{subarray}{c} \sigma\in\cJ_k \\ \sigma'\in\cJ'_{d-p-k} \end{subarray}}
      |\vec L_\sigma\wedge\vec E_{\sigma'}\wedge\vec Z|_\Theta\leq t^{1-(k-k_0)(1+\gamma)},\qquad k=k_0,\ldots,k_1.
    \end{equation}
    Writing
    \[ \vec Z=\sum_{j=\max(0,p-n)}^{\min(p,m)}\sum_{\begin{subarray}{c} \rho\in\cJ_j \\ \rho'\in\cJ'_{p-j} \end{subarray}}
       Z_{\rho,\rho'}\vec L_\rho\wedge\vec E_{\rho'}, \]
    we see that \eqref{eq:ba_Thetanized} for each $k$ means exactly that
    \[ \ Z_{\rho,\rho'}=0,\qquad\qquad\qquad\quad\ \text{ if }\rho\in\cJ_j,\ j>m-k, \]
    \begin{equation} \label{eq:coordinates_filtered}
      |Z_{\rho,\rho'}|\leq t^{1-(k-k_0)(1+\gamma)},\qquad\text{ if }\rho\in\cJ_j,\ j\leq m-k.
    \end{equation}
    Hence we see that all the inequalities in \eqref{eq:ba_Thetanized} with $k>k_1$ are trivial. Next, since we are concerned about large values of $t$, by Minkowski's first convex body theorem we may confine ourselves to considering only positive values of $1+\gamma$. Then the function $t^{1-(k-k_0)(1+\gamma)}$ is non-increasing with respect to $k$, so for each $\rho\in\cJ_j$ of all the inequalities \eqref{eq:coordinates_filtered} we may keep the ones with the largest $k$, i.e. with the one equal to $m-j$. Thus, \eqref{eq:ba_Thetanized} becomes equivalent to
    \begin{equation} \label{eq:coordinates_graduated}
      |Z_{\rho,\rho'}|\leq t^{1-(k-k_0)(1+\gamma)},\qquad\text{ if }\rho\in\cJ_{m-k},\ k_0\leq k\leq k_1.\phantom{,\ \rho'\in\cJ'_{p-m+k}}
    \end{equation}
    On the other hand, \eqref{eq:ba_modified_Thetanized} means that
    \begin{equation} \label{eq:coordinates_sandwiched}
      |Z_{\rho,\rho'}|\leq t^{1-(k-k_0)(1+\gamma)},\qquad\text{ if }\rho\in\cJ_{m-k},\ \rho'\in\cJ'_{p-m+k},\ k_0\leq k\leq k_1,
    \end{equation}
    which is obviously equivalent to \eqref{eq:coordinates_graduated}.
  \end{proof}

  For $p=1$ Definition \ref{def:ba} coincides with Definition \ref{def:belpha_1}, i.e. $\beta_1=\gb_1$ and $\alpha_1=\ga_1$. It is seen from the following

  \begin{proposition}
    The quantity $\beta_1$ (resp. $\alpha_1$) equals the supremum of the real numbers $\gamma$, such that there are arbitrarily large values of $t$ for which (resp. such that for every $t$ large enough)
    the system of inequalities
    \begin{equation} \label{eq:belpha_1_prop}
      |\vec z|\leq t,\qquad|\vec L\wedge\vec z|\leq t^{-\gamma},
    \end{equation}
    where $\vec L=\pmb\ell_1\wedge\ldots\wedge\pmb\ell_m$,
    has a nonzero solution in $\vec z\in\Z^d$.
  \end{proposition}

  \begin{proof}
    The parallelepiped in $\R^d$ defined by \eqref{eq:belpha_1_definition} can be written as
    \begin{equation*}
      M_\gamma(t)=\Big\{ \vec z\in\R^d \,\Big|\,
      \max_{1\leq j\leq m}|\langle\vec e_j,\vec z\rangle|\leq t,\
      \max_{1\leq i\leq n}|\langle\pmb\ell_{m+i},\vec z\rangle|\leq t^{-\gamma}
      \Big\},
    \end{equation*}
    where $\langle\,\cdot\,,\cdot\,\rangle$ denotes the inner product in $\R^d$.

    The vectors $\pmb\ell_{m+1},\ldots,\pmb\ell_d$ form a basis of the orthogonal complement of $\cL$. Therefore, since the Euclidean norm of $\vec L\wedge\vec z$ equals the $(m+1)$-dimensional volume of the parallelepiped spanned by $\pmb\ell_1,\ldots,\pmb\ell_m,\vec z$, we have
    \[ |\vec L\wedge\vec z|\asymp\max_{1\leq i\leq n}|\langle\pmb\ell_{m+i},\vec z\rangle|, \]
    with the implied constant depending only on $\Theta$. Besides that,
    \[ |\vec z|\asymp\max\Big(\max_{1\leq j\leq m}|\langle\vec e_j,\vec z\rangle|,\,\max_{1\leq i\leq n}|\langle\pmb\ell_{m+i},\vec z\rangle|\ \Big) \]
    where the implied constant again depends only on $\Theta$.

    Hence there is a positive constant $c$ depending only on $\Theta$, such that the set $M'_\gamma(t)$ defined by \eqref{eq:belpha_1_prop} satisfies the relation
    \[ c^{-1}M_\gamma(t)\subseteq M'_\gamma(t)\subseteq cM_\gamma(t), \]
    at least for $t\geq1$, $\gamma\geq0$. Which immediately implies the desired result.
  \end{proof}

  \section{Known transference inequalities} \label{sec:known}

  The \emph{transference principle} connects the problem of approximating the space of solutions to \eqref{eq:the_system} to the analogous problem for the system
  \begin{equation} \label{eq:transposed_system}
    \tr\Theta\vec y=\vec x,
  \end{equation}
  where $\tr\Theta$ denotes the transpose of $\Theta$. Let us denote the intermediate Diophantine exponents corresponding to $\tr\Theta$ by $\beta_p^\ast$, $\alpha_p^\ast$, $\gb_p^\ast$, $\ga_p^\ast$.

  The classical transference inequalities estimating $\gb_1$ in terms of $\gb_1^\ast$, and $\ga_1$ in terms of $\ga_1^\ast$ belong to A.~Ya.~Khintchine, V.~Jarn\'{\i}k, F.~Dyson, and A.~Apfelbeck. We remind that, as we showed in the end of the previous Section, $\beta_1=\gb_1$, $\alpha_1=\ga_1$, $\beta_1^\ast=\gb_1^\ast$, $\alpha_1^\ast=\ga_1^\ast$.

  \subsection{Regular exponents}

  In \cite{khintchine_palermo} A.~Ya.~Khintchine proved for $m=1$ his famous transference inequalities
  \begin{equation} \label{eq:khintchine_transference}
    \gb_1^\ast\geq n\gb_1+n-1,\qquad
    \gb_1\geq\frac{\gb_1^\ast}{(n-1)\gb_1^\ast+n}\,,
  \end{equation}
  which were generalized later by F.~Dyson \cite{dyson}, who proved that for arbitrary $n$, $m$
  \begin{equation} \label{eq:dyson_transference}
    \gb_1^\ast\geq\frac{n\gb_1+n-1}{(m-1)\gb_1+m}\,.
  \end{equation}
  While \eqref{eq:khintchine_transference} cannot be improved (see \cite{jarnik_1936_1}, \cite{jarnik_1936_2}) if only $\gb_1$ and $\gb_1^\ast$ are considered, stronger inequalities can be obtained if $\ga_1$ and $\ga_1^\ast$ are also taken into account. The corresponding result for $m=1$ belongs to M.~Laurent and Y.~Bugeaud (see \cite{laurent_2dim}, \cite{bugeaud_laurent_up_down}). They proved that if the system \eqref{eq:the_system} has no non-zero integer solutions, then
  \begin{equation} \label{eq:bugeaud_laurent}
    \frac{(\ga_1^\ast-1)\gb_1^\ast}{((n-2)\ga_1^\ast+1)\gb_1^\ast+(n-1)\ga_1^\ast}\leq
    \gb_1\leq
    \frac{(1-\ga_1)\gb_1^\ast-n+2-\ga_1}{n-1}\,.
  \end{equation}
  The inequalities \eqref{eq:bugeaud_laurent} were generalized to the case of arbitrary $n$, $m$ by the author in \cite{german_JNT}, where it was proved for arbitrary $n$, $m$ that if the space of integer solutions of \eqref{eq:the_system} is not a one-dimensional lattice, then along with \eqref{eq:dyson_transference} we have
  \begin{align}
    & \gb_1^\ast\geq
    \frac{(n-1)(1+\gb_1)-(1-\ga_1)}{(m-1)(1+\gb_1)+(1-\ga_1)}\,, \label{eq:loranoyadenie_2} \\
    & \gb_1^\ast\geq
    \frac{(n-1)(1+\gb_1^{-1})-(\ga_1^{-1}-1)}{(m-1)(1+\gb_1^{-1})+(\ga_1^{-1}-1)}
    \,, \label{eq:loranoyadenie_3}
  \end{align}
  with \eqref{eq:loranoyadenie_2} stronger than \eqref{eq:loranoyadenie_3} if and only if $\ga_1<1$.

  \subsection{Uniform exponents}

  V.~Jarn\'{\i}k and A.~Apfelbeck proved literal analogues of \eqref{eq:khintchine_transference} and \eqref{eq:dyson_transference} for the uniform exponents, i.e. with $\gb_1$, $\gb_1^\ast$ replaced by $\ga_1$, $\ga_1^\ast$, respectively (see \cite{jarnik_tiflis}, \cite{apfelbeck}). They also obtained some stronger inequalities of a more cumbersome appearance. Among them, lonely in its elegance, stands the \emph{equality}
  \begin{equation} \label{eq:jarnik_equality}
    \ga_1^{-1}+\ga_1^\ast=1
  \end{equation}
  proved by Jarn\'{\i}k for $n=1$, $m=2$. The results of Jarn\'{\i}k and Apfelbeck were improved by the author in \cite{german_JNT}, where it was shown that for arbitrary $n$, $m$ we have
  \begin{equation} \label{eq:my_inequalities_cases}
    \ga_1^\ast\geq
    \begin{cases}
      \dfrac{n-1}{m-\ga_1},\quad\ \ \text{ if }\ \ga_1\leq1, \\
      \dfrac{n-\ga_1^{-1\vphantom{\big|}}}{m-1},\quad\ \ \text{if }\ \ga_1\geq1.
    \end{cases}
  \end{equation}

  \subsection{Khintchine's inequalities split}

  Laurent and Bugeaud used the exponents $\gb_p$ to split \eqref{eq:khintchine_transference} into a chain of inequalities relating  $\gb_p$ to $\gb_{p+1}$. Namely, they proved that for $m=1$ we have  $\gb_1^\ast=\gb_n$ and
  \begin{equation} \label{eq:khintchine_transference_split}
    \gb_{p+1}\geq\frac{(n-p+1)\gb_{p}+1}{n-p}\,,\qquad
    \gb_{p}\geq\frac{p\gb_{p+1}}{\gb_{p+1}+p+1}\,,\qquad p=1,\ldots,n-1.
  \end{equation}
  Besides that, they proved for $m=1$ that if the system \eqref{eq:the_system} has no non-zero integer solutions, then we have  $\ga_1^\ast=\ga_n$ and
  \begin{equation} \label{eq:laurent_inter_mixed}
    \gb_2\geq\frac{\gb_1+\ga_1}{1-\ga_1}\,,\qquad
    \gb_{n-1}\geq\frac{1-\ga_n^{-1}}{\gb_n^{-1}+\ga_n^{-1}}\,,
  \end{equation}
  which, combined with \eqref{eq:khintchine_transference_split}, gave them \eqref{eq:bugeaud_laurent}.

  \section{Main results for intermediate Diophantine exponents} \label{sec:results}

  In this paper we generalize \eqref{eq:khintchine_transference_split} and its analogue for the uniform exponents to the case of arbitrary $n$, $m$. We show (see Proposition \ref{prop:starred_equals_nonstarred} in Section \ref{sec:transposed}) that
  \begin{equation} \label{eq:ab_via_ab_transposed}
    \gb_p^\ast=\gb_{d-p},\qquad\ga_p^\ast=\ga_{d-p},\qquad p=1,\ldots,d-1,
  \end{equation}
  and prove

  \begin{theorem} \label{t:inter_dyfel}
    For each $p=1,\ldots,d-2$ the following statements hold.

    If $p\geq m$, then
    \begin{equation} \label{eq:inter_dyson_p_geq}
      (d-p-1)(1+\gb_{p+1})\geq(d-p)(1+\gb_p),
    \end{equation}
    \begin{equation} \label{eq:inter_apfel_p_geq}
      (d-p-1)(1+\ga_{p+1})\geq(d-p)(1+\ga_p).
    \end{equation}

    If $p\leq m-1$, then
    \begin{equation} \label{eq:inter_dyson_p_leq}
      (d-p-1)(1+\gb_p)^{-1}\geq(d-p)(1+\gb_{p+1})^{-1}-n,
    \end{equation}
    \begin{equation} \label{eq:inter_apfel_p_leq}
      (d-p-1)(1+\ga_p)^{-1}\geq(d-p)(1+\ga_{p+1})^{-1}-n.
    \end{equation}
  \end{theorem}

  The second result of the current paper generalizes \eqref{eq:laurent_inter_mixed}. We prove

  \begin{theorem} \label{t:inter_loranoyadenie}
    Suppose that the space of integer solutions of \eqref{eq:the_system} is not a one-dimensional lattice. Then for $m=1$ we have
    \begin{equation} \label{eq:inter_loranoyadenie_m_is_1}
      \gb_2\geq\frac{\gb_1+\ga_1}{1-\ga_1}\,,
    \end{equation}
    and for $m\geq2$ we have
    \begin{equation} \label{eq:inter_loranoyadenie_m_geq_2}
      \gb_2\geq
      \begin{cases}
        \dfrac{\ga_1-1}{2+\gb_1-\ga_1}\,,\quad\text{ if }\ \ga_1\neq\infty, \\
        \dfrac{1-\ga_1^{-1} \vphantom{\frac{\big|}{}} }{\gb_1^{-1}+\ga_1^{-1}}\,.
      \end{cases}
    \end{equation}
  \end{theorem}

  The inequality of \eqref{eq:inter_loranoyadenie_m_is_1} is exactly the first inequality of \eqref{eq:laurent_inter_mixed}. The second inequality of \eqref{eq:inter_loranoyadenie_m_geq_2} in view of \eqref{eq:ab_via_ab_transposed} gives the second inequality of \eqref{eq:laurent_inter_mixed}.

  It follows from Theorem \ref{t:inter_dyfel} that for $m\geq2$
  \begin{equation}
    (d-2)(1+\gb_{d-1})^{-1}\leq(1+\gb_2)^{-1}+m-2.
  \end{equation}
  Combining this inequality with \eqref{eq:inter_loranoyadenie_m_geq_2} we get \eqref{eq:loranoyadenie_2} and \eqref{eq:loranoyadenie_3}, in case $m\geq2$.

  The third result of this paper splits the inequalities \eqref{eq:my_inequalities_cases}. It is the following

  \begin{theorem} \label{t:inter_my_inequalities}
    For $m=1$ we have
    \begin{equation} \label{eq:inter_my_inequalities_m_is_1}
      \ga_2\geq(1-\ga_1)^{-1}-\frac{n-2}{n-1}\,.
    \end{equation}
    For $m\geq2$ we have
    \begin{equation} \label{eq:inter_my_inequalities_m_geq_2}
      \ga_2\geq
      \begin{cases}
        \dfrac{n-1}{-n-(d-2)(1-\ga_1)^{-1}}\,,\quad\text{ if }\ \ga_1\leq1, \\
        \dfrac{m-1\vphantom{\frac{|}{}}}{\phantom{-}n+(d-2)(\ga_1-1)^{-1}}\,,\quad\text{ if }\ \ga_1\geq1.
      \end{cases}
    \end{equation}
  \end{theorem}

  Let us show that Theorem \ref{t:inter_my_inequalities} splits \eqref{eq:my_inequalities_cases} the very same way Theorem \ref{t:inter_loranoyadenie} splits \eqref{eq:loranoyadenie_2} and \eqref{eq:loranoyadenie_3}. It follows from Theorem \ref{t:inter_dyfel} that for $m=1$
  \begin{equation} \label{eq:apfel_shampur_m_is_1}
    1+\ga_n\geq(n-1)(1+\ga_2)
  \end{equation}
  and that for $m\geq2$
  \begin{equation} \label{eq:apfel_shampur_m_geq_2}
    (d-2)(1+\ga_{d-1})^{-1}\leq(1+\ga_2)^{-1}+m-2.
  \end{equation}
  Combining \eqref{eq:apfel_shampur_m_geq_2} with \eqref{eq:inter_my_inequalities_m_geq_2}, we get \eqref{eq:my_inequalities_cases} for  $m\geq2$. As for $m=1$, we always have $\ga_1\leq1$ in this case, so \eqref{eq:apfel_shampur_m_is_1} and \eqref{eq:inter_my_inequalities_m_is_1} indeed gives \eqref{eq:my_inequalities_cases} with $m=1$.

  \section{Schmidt--Summerer's exponents} \label{sec:schmex}

  Let $\La$ be a unimodular $d$-dimensional lattice in $\R^d$. Denote by $\cB_\infty^d$ the unit ball in sup-norm, i.e. the cube with vertices at the points $(\pm1,\ldots,\pm1)$. For each $d$-tuple $\pmb\tau=(\tau_1,\ldots,\tau_d)\in\R^d$ denote by $D_{\pmb\tau}$ the diagonal $d\times d$ matrix with $e^{\tau_1},\ldots,e^{\tau_d}$ on the main diagonal.
  Let us also denote by $\lambda_p(M)$ the $p$-th successive minimum of a compact symmetric convex body $M\subset\R^d$ (centered at the origin) with respect to the lattice $\La$.


  Suppose we have a path $\gT$ in $\R^d$ defined as $\pmb\tau=\pmb\tau(s)$, $s\in\R_+$, such that
  \begin{equation} \label{eq:sum_of_taus_is_zero}
    \tau_1(s)+\ldots+\tau_d(s)=0,\quad\text{ for all }s.
  \end{equation}
  In our further applications to Diophantine approximation we shall confine ourselves to a path that is a ray with the endpoint at the origin and all the functions $\tau_1(s),\ldots,\tau_d(s)$ being linear. However, in this Section, as well as in the next one, all the definitions and statements are given for arbitrary paths and lattices.

  Set $\cB(s)=D_{\pmb\tau(s)}\cB_\infty^d$. For each $p=1,\ldots,d$\, let us consider the functions
  \[ \psi_p(\La,\gT,s)=\frac{\ln(\lambda_p(\cB(s)))}{s}\,,\qquad \Psi_p(\La,\gT,s)=\sum_{i=1}^p\psi_i(\La,\gT,s). \]

  \begin{definition} \label{def:schmidt_psi}
    We call the quantities
    \[ \bpsi_p(\La,\gT)=\liminf_{s\to+\infty}\psi_p(\La,\gT,s),\qquad
       \apsi_p(\La,\gT)=\limsup_{s\to+\infty}\psi_p(\La,\gT,s) \]
    \emph{the $p$-th lower} and \emph{upper Schmidt--Summerer's exponents of the first type}, respectively.
  \end{definition}

  \begin{definition} \label{def:schmidt_Psi}
    We call the quantities
    \[ \bPsi_p(\La,\gT)=\liminf_{s\to+\infty}\Psi_p(\La,\gT,s)\,,\qquad
       \aPsi_p(\La,\gT)=\limsup_{s\to+\infty}\Psi_p(\La,\gT,s) \]
    \emph{the $p$-th lower} and \emph{upper Schmidt--Summerer's exponents of the second type}, respectively.
  \end{definition}

  Sometimes, when it is clear from the context what lattice and what path are under consideration, we shall write simply $\psi_p(s)$, $\Psi_p(s)$, $\bpsi_p$, $\apsi_p$, $\bPsi_p$, and $\aPsi_p$.

  The following Proposition and its Corollaries generalize some of the observations made in the papers \cite{schmidt_summerer} and \cite{bugeaud_laurent_up_down}.

  \begin{proposition} \label{prop:mink}
    For any $\La$ and $\gT$ we have
    \begin{equation} \label{eq:mink}
      0\leq-\Psi_d(s)=O(s^{-1}).
    \end{equation}
    Particularly,
    \begin{equation} \label{eq:mink_0}
      \bPsi_d=\aPsi_d=0.
    \end{equation}
  \end{proposition}

  \begin{proof}
    Due to \eqref{eq:sum_of_taus_is_zero} the volumes of all the parallelepipeds $\cB(s)$ are equal to $2^d$, so by Minkowski's second theorem we have
    \[ \frac{1}{d!}\leq\prod_{i=1}^d\lambda_i(\cB(s))\leq1. \]
    Hence
    \[ -\frac{\ln(d!)}{s}\leq\sum_{i=1}^d\psi_i(s)\leq0, \]
    which immediately implies \eqref{eq:mink}.
  \end{proof}

  \begin{corollary} \label{cor:precise_inter_dyfel}
    For every $p$ within the range $1\leq p\leq d-2$ and every $s>0$ we have
    \begin{equation} \label{eq:Psi_sandwich}
      \frac{p+1}{p}\Psi_p(s)\leq\Psi_{p+1}(s)\leq\frac{d-p-1}{d-p}\Psi_p(s).
    \end{equation}
  \end{corollary}

  \begin{proof}
    In view of \eqref{eq:mink}, it follows from the inequalities $\psi_i(s)\leq\psi_{i+1}(s)$, $i=1,\ldots,d-1$, that
    \[ \frac 1p\sum_{i=1}^p\psi_i(s)\leq\psi_{p+1}(s)\leq\frac{-1}{d-p}\sum_{i=1}^p\psi_i(s), \]
    which immediately implies \eqref{eq:Psi_sandwich}.
  \end{proof}

  Taking the $\liminf$ and the $\limsup$ of all the sides of \eqref{eq:Psi_sandwich}, we get

  \begin{corollary} \label{cor:Psi_inter_dyson}
    For any $\La$ and $\gT$ and any $p$ within the range $1\leq p\leq d-2$ we have
    \begin{equation} \label{eq:Psi_inter_dyson}
      \frac{p+1}{p}\bPsi_p\leq\bPsi_{p+1}\leq\frac{d-p-1}{d-p}\bPsi_p
      \qquad\text{ and }\qquad
      \frac{p+1}{p}\aPsi_p\leq\aPsi_{p+1}\leq\frac{d-p-1}{d-p}\aPsi_p.
    \end{equation}
  \end{corollary}

  Applying consequently \eqref{eq:Psi_inter_dyson} we get

  \begin{corollary} \label{cor:Psi_dyson}
    For any $\La$ and $\gT$ we have
    \begin{equation} \label{eq:Psi_dyson}
      (d-1)\bPsi_1\leq\bPsi_{d-1}\leq\frac{\bPsi_1}{d-1}
      \qquad\text{ and }\qquad
      (d-1)\aPsi_1\leq\aPsi_{d-1}\leq\frac{\aPsi_1}{d-1}.
    \end{equation}
  \end{corollary}

  Another simple corollary to Proposition \ref{prop:mink} is the following statement.

  \begin{corollary} \label{cor:Psi_and_psi}
    For any $\La$ and $\gT$ we have
    \begin{equation} \label{eq:Psi_and_psi}
      \bPsi_{d-1}=-\apsi_d\qquad\text{ and }\qquad\aPsi_{d-1}=-\bpsi_d.
    \end{equation}
  \end{corollary}

  As we shall see later, the first inequalities of \eqref{eq:Psi_dyson} generalize Khintchine's and Dyson's transference inequalities.

  \section{Schmidt--Summerer's exponents of the second type from the point of view of multilinear algebra} \label{sec:multi_schmex}

  As before, let us consider the space $\wedge^p(\R^d)$ as the $\binom dp$-dimensional Euclidean space with the orthonormal basis consisting of the multivectors
  \[ \vec e_{i_1}\wedge\ldots\wedge\vec e_{i_p},\qquad 1\leq i_1<\ldots<i_p\leq d. \]
  Let us order the set of the $p$-element subsets of $\{1,\ldots,d\}$ lexicographically and denote the $j$-th subset by $\sigma_j$. To each $d$-tuple $\pmb\tau=(\tau_1,\ldots,\tau_d)$ let us associate the $r$-tuple
  \begin{equation} \label{eq:T_hat}
    \widehat{\pmb\tau}=\Big(\widehat\tau_1,\ldots,\widehat\tau_r\Big),\qquad\widehat\tau_j=\sum_{i\in\sigma_j}\tau_i,\qquad r={\binom dp}.
  \end{equation}
  Thus, a path $\gT:s\to\pmb\tau(s)$ leads us by \eqref{eq:T_hat} to the path $\widehat\gT:s\to\widehat{\pmb\tau}(s)$ also satisfying the condition
  \[ \widehat\tau_1(s)+\ldots+\widehat\tau_r(s)=0,\quad\text{ for all }s. \]
  Finally, given a lattice $\La\subset\R^d$, let us associate to it the lattice $\widehat\La=\wedge^p(\La)$.

  \begin{proposition} \label{prop:Psi_p_is_Psi_1}
    For any $\La$ and $\gT$ we have
    \[ \bPsi_p(\La,\gT)=\bPsi_1(\widehat\La,\widehat\gT)=\bpsi_1(\widehat\La,\widehat\gT)
    \quad\text{ and }\quad
    \aPsi_p(\La,\gT)=\aPsi_1(\widehat\La,\widehat\gT)=\apsi_1(\widehat\La,\widehat\gT). \]
  \end{proposition}

  \begin{proof}
    Let us denote by $\lambda_i(M)$ the $i$-th successive minimum of a body $M$ with respect to $\La$ if $M\subset\R^d$ and with respect to $\widehat\La$ if $M\in\wedge^p(\R^d)$.

    The matrix $D_{\widehat{\pmb\tau}}$ is the $p$-th compound of $D_{\pmb\tau}$:
    \[ D_{\widehat{\pmb\tau}}=D_{\pmb\tau}^{(p)}. \]
    This means that $D_{\widehat{\pmb\tau}}\cB_\infty^r$ is comparable to Mahler's $p$-th compound convex body of $D_{\pmb\tau}\cB_\infty^d$ (see \cite{mahler_compound_I}), i.e. there is a positive constant $c$ depending only on $d$, such that
    \begin{equation} \label{eq:comparable_to_pseudo_compound}
      c^{-1}D_{\widehat{\pmb\tau}}\cB_\infty^r\subset[D_{\pmb\tau}\cB_\infty^d]^{(p)}\subset cD_{\widehat{\pmb\tau}}\cB_\infty^r.
    \end{equation}
    In \cite{schmidt_DA} the set $D_{\widehat{\pmb\tau}}\cB_\infty^r$ is called the $p$-th pseudo-compound parallelepiped for $D_{\pmb\tau}\cB_\infty^d$.

    It follows from Mahler's theory of compound bodies that
    \begin{equation} \label{eq:first_minimum_vs_product}
      \lambda_1\left([D_{\pmb\tau}\cB_\infty^d]^{(p)}\right)\asymp\prod_{i=1}^p\lambda_i\left(D_{\pmb\tau}\cB_\infty^d\right)
    \end{equation}
    with the implied constants depending only on $d$. Combining \eqref{eq:comparable_to_pseudo_compound} and \eqref{eq:first_minimum_vs_product} we get
    \[ \ln\left(\lambda_1\left(D_{\widehat{\pmb\tau}(s)}\cB_\infty^r\right)\right)=
       \sum_{i=1}^p\ln\left(\lambda_i\left(D_{\pmb\tau(s)}\cB_\infty^d\right)\right)+O(1), \]
    whence
    \[ \psi_1(\widehat\La,\widehat\gT,s)=\sum_{i=1}^p\psi_i(\La,\gT,s)+o(1). \]
    It remains to take the $\liminf$ and the $\limsup$ of both sides as $s\to\infty$.
  \end{proof}

  \section{Diophantine exponents in terms of Schmidt--Summerer's exponents} \label{sec:ex_via_schmex}

  Let $\pmb\ell_1,\ldots,\pmb\ell_d$, $\vec e_1,\ldots,\vec e_d$ be as in Section \ref{sec:laurexp}. Set
  \[ T=
     \begin{pmatrix}
       E_m & 0 \\
       \Theta & E_n
     \end{pmatrix}. \]
  Then
  \[ \tr{(T^{-1})}=
     \begin{pmatrix}
       E_m & -\tr\Theta \\
       0 & E_n
     \end{pmatrix}, \]
  so the bases $\pmb\ell_1,\ldots,\pmb\ell_m,\vec e_{m+1},\ldots,\vec e_d$ and $\vec e_1,\ldots,\vec e_m,\pmb\ell_{m+1},\ldots,\pmb\ell_d$ are dual.

  Let us specify a lattice $\La$ and a path $\gT$ as follows. Set
  \begin{equation} \label{eq:La}
    \La=T^{-1}\Z^d=\Big\{ \tr{\Big(\langle\vec e_1,\vec z\rangle,\ldots,\langle\vec e_m,\vec z\rangle,\langle\pmb\ell_{m+1},\vec z\rangle,\ldots,\langle\pmb\ell_d,\vec z\rangle\Big)}\in\R^d \,\Big|\, \vec z\in\Z^d \Big\}
  \end{equation}
  and define $\gT:s\mapsto\pmb\tau(s)$ by
  \begin{equation} \label{eq:path}
    \tau_1(s)=\ldots=\tau_m(s)=s,\quad\tau_{m+1}(s)=\ldots=\tau_d(s)=-ms/n.
  \end{equation}
  Schmidt--Summerer's exponents $\bpsi_p$, $\apsi_p$ corresponding to such $\La$ and $\gT$ and the exponents $\beta_p$, $\alpha_p$ are but two different points of view at the same phenomenon. The same can be said about $\bPsi_p$, $\aPsi_p$ and $\gb_p$, $\ga_p$. It is exposed in the following two Propositions.

  \begin{proposition} \label{prop:belpha_via_psis}
    We have
    \begin{equation} \label{eq:belpha_via_psis}
      (1+\beta_p)(1+\bpsi_p)=(1+\alpha_p)(1+\apsi_p)=d/n.
    \end{equation}
  \end{proposition}

  \begin{proof}
    The parallelepiped in $\R^d$ defined by \eqref{eq:belpha_1_definition} can be written as
    \begin{equation*}
      M_\gamma(t)=\Big\{ \vec z\in\R^d \,\Big|\,
      \max_{1\leq j\leq m}|\langle\vec e_j,\vec z\rangle|\leq t,\
      \max_{1\leq i\leq n}|\langle\pmb\ell_{m+i},\vec z\rangle|\leq t^{-\gamma}
        \Big\},
    \end{equation*}
    where $\langle\,\cdot\,,\cdot\,\rangle$ is the inner product in $\R^d$.

    Therefore, $\beta_p$ (resp. $\alpha_p$) equals the supremum of the real numbers $\gamma$, such that there are arbitrarily large values of $t$ for which (resp. such that for every $t$ large enough) the parallelepiped $M_\gamma(t)$ contains $p$ linearly independent integer points.

    Hence, considering the parallelepipeds
    \begin{equation} \label{eq:P_gamma}
      P_\gamma(t)=T^{-1}M_\gamma(t)=\Big\{ \vec z\in\R^d \,\Big|\,
      \max_{1\leq j\leq m}|\langle\vec e_j,\vec z\rangle|\leq t,\
      \max_{1\leq i\leq n}|\langle\vec e_{m+i},\vec z\rangle|\leq t^{-\gamma}
      \Big\},
    \end{equation}
    we see that
    \begin{equation} \label{eq:belpha_p_via_parallelepipeds}
      \beta_p=\limsup_{t\to+\infty}\big\{ \gamma\in\R \,\big|\, \lambda_p(P_\gamma(t))=1 \big\}\,,\qquad\alpha_p=\liminf_{t\to+\infty}\big\{ \gamma\in\R \,\big|\, \lambda_p(P_\gamma(t))=1 \big\},
    \end{equation}
    where $\lambda_p(P_\gamma(t))$ is the $p$-th minimum of $P_\gamma(t)$ with respect to $\La$.

    But $P_{m/n}(t)=D_{\pmb\tau(\ln t)}\cB_\infty^d$, so
    \begin{equation} \label{eq:psis_via_parallelepipeds}
      \bpsi_p(\La,\gT)=\liminf_{t\to+\infty}\frac{\ln(\lambda_p(P_{m/n}(t)))}{\ln t}\,,\qquad\apsi_p(\La,\gT)=\limsup_{t\to+\infty}\frac{\ln(\lambda_p(P_{m/n}(t)))}{\ln t}\,.
    \end{equation}

    A simple calculation shows that
    \[ P_\gamma(t)=t^{\frac{m-n\gamma}{d}}P_{m/n}\big(t^{\frac{n+n\gamma}{d}}\big), \]
    i.e.
    \[ \lambda_p(P_\gamma(t))=(t')^{\frac{-m+n\gamma}{n+n\gamma}}\lambda_p\big(P_{m/n}(t')\big) \]
    with $t'=t^{\frac{n+n\gamma}{d}}$. Therefore, the equality
    \[ \lambda_p(P_\gamma(t))=1 \]
    holds if and only if
    \[ 1-\frac{d}{n+n\gamma}+\frac{\ln(\lambda_p(P_{m/n}(t')))}{\ln t'}=0. \]
    Hence, in view of \eqref{eq:belpha_p_via_parallelepipeds}, \eqref{eq:psis_via_parallelepipeds}, we get
    \[ \beta_p=\limsup_{t\to+\infty}\left\{ \frac dn\left( 1+\frac{\ln(\lambda_p(P_{m/n}(t)))}{\ln t} \right)^{-1}-1 \right\}=\frac dn\left(1+\bpsi_p\right)^{-1}-1 \]
    and
    \[ \alpha_p=\liminf_{t\to+\infty}\left\{ \frac dn\left( 1+\frac{\ln(\lambda_p(P_{m/n}(t)))}{\ln t} \right)^{-1}-1 \right\}=\frac dn\left(1+\apsi_p\right)^{-1}-1, \]
    which immediately implies \eqref{eq:belpha_via_psis}.
  \end{proof}

  \begin{proposition} \label{prop:ba_via_Psis}
    Set $\varkappa_p=\min(p,\frac mn(d-p))$. Then
    \begin{equation} \label{eq:ba_via_Psis}
      (1+\gb_p)(\varkappa_p+\bPsi_p)=(1+\ga_p)(\varkappa_p+\aPsi_p)=d/n.
    \end{equation}
  \end{proposition}

  \begin{proof}
    Let $\vec L_\sigma$, $\vec E_\sigma$, $\cJ_k$, $\cJ'_k$ be as in Section \ref{sec:laurexp}.

    Since $T^{-1}\pmb\ell_i=\vec e_i$ and $T^{-1}\vec e_j=\vec e_j$, if $1\leq i\leq m$ and $m+1\leq j\leq d$, we have
    \begin{equation} \label{eq:T_deapplied_to_LE}
      (T^{-1})^{(k+k')}(\vec L_\sigma\wedge\vec E_{\sigma'})=\vec E_\sigma\wedge\vec E_{\sigma'},\qquad\text{ for each }
      \sigma\in\cJ_k,\ \sigma'\in\cJ'_{k'},
    \end{equation}
    where $(T^{-1})^{(k+k')}$ is the $(k+k')$-th compound of $T^{-1}$. Furthermore, since $\La=T^{-1}\Z^d$, we have
    \begin{equation} \label{eq:T_and_La}
      \widehat\La=\wedge^p(\La)=(T^{-1})^{(p)}(\wedge^p(\Z^d)).
    \end{equation}
    Hence for each $\vec Z\in\wedge^p(\Z^d)$ and each $\sigma\in\cJ_k$, $\sigma'\in\cJ'_{d-p-k}$ (with $k$ satisfying $k_0\leq k\leq k_1$) we get
    \begin{equation} \label{eq:T_deapplied_to_LEZ}
      |\vec L_\sigma\wedge\vec E_{\sigma'}\wedge\vec Z|=
      |(T^{-1})^{(d-p)}(\vec L_\sigma\wedge\vec E_{\sigma'})\wedge(T^{-1})^{(p)}\vec Z|=
      |\vec E_\sigma\wedge\vec E_{\sigma'}\wedge\vec Z'|,
    \end{equation}
    where $\vec Z'\in\widehat\La$. Here, besides \eqref{eq:T_deapplied_to_LE}, \eqref{eq:T_and_La}, we have made use of the fact that for every $\vec V\in\wedge^p(\R^d)$, $\vec W\in\wedge^{d-p}(\R^d)$ the wedge product $\vec V\wedge\vec W$ is a real number and
    \[ |\vec V\wedge\vec W|=|T^{(p)}\vec V\wedge T^{(d-p)}\vec W|, \]
    provided $\det T=1$.

    We conclude from \eqref{eq:T_deapplied_to_LEZ} and Proposition \ref{prop:ba_substitution} that $\gb_p$ (resp. $\ga_p$) equals the supremum of the real numbers $\gamma$, such that there are arbitrarily large values of $t$ for which (resp. such that for every $t$ large enough) the system of inequalities
    \begin{equation} \label{eq:ba_modified_with_T_applied}
      \max_{\begin{subarray}{c} \sigma\in\cJ_k \\ \sigma'\in\cJ'_{d-p-k} \end{subarray}}
      |\vec E_\sigma\wedge\vec E_{\sigma'}\wedge\vec Z|\leq t^{1-(k-k_0)(1+\gamma)},\qquad k=k_0,\ldots,k_1,
    \end{equation}
    has a nonzero solution in $\vec Z\in\widehat\La$.

    The inequalities \eqref{eq:ba_modified_with_T_applied} define the parallelepiped
    \begin{equation} \label{eq:P_gamma_hat}
      \widehat P_\gamma(t)=\Big\{ \vec Z\in\wedge^p(\R^d) \,\Big|\,
      \max_{\begin{subarray}{c} \sigma\in\cJ_{m-k} \\ \sigma'\in\cJ'_{p-m+k} \end{subarray}}
      |\langle\vec E_\sigma\wedge\vec E_{\sigma'},\vec Z\rangle|\leq t^{1-(k-k_0)(1+\gamma)},\
      k=k_0,\ldots,k_1 \Big\},
    \end{equation}
    where $\langle\,\cdot\,,\cdot\,\rangle$ is the inner product in $\wedge^p(\R^d)$.
    By analogy with \eqref{eq:belpha_p_via_parallelepipeds} we can write
    \begin{equation} \label{eq:ba_via_parallelepipeds}
      \gb_p=\limsup_{t\to+\infty}\Big\{ \gamma\in\R \ \Big|\, \lambda_1\big(\widehat P_\gamma(t)\big)=1 \Big\}\,,\qquad
      \ga_p=\liminf_{t\to+\infty}\Big\{ \gamma\in\R \ \Big|\, \lambda_1\big(\widehat P_\gamma(t)\big)=1 \Big\},
    \end{equation}
    where $\lambda_1\big(\widehat P_\gamma(t)\big)$ is the first minimum of $\widehat P_\gamma(t)$ with respect to $\widehat\La$.

    Consider the path $\widehat\gT$ defined by \eqref{eq:T_hat} for $\gT$. Then
    \[ \widehat\tau_j(s)=\sum_{i\in\sigma_j}\tau_i(s), \]
    and if $\sigma_j\cap\{1,\ldots,m\}\in\cJ_{m-k}$\,, we have
    \[ \widehat\tau_j(s)=(m-k)s-\frac{(p-(m-k))m}{n}s=\left(\frac dn(k_0-k)+\varkappa_p\right)s=(1-(k-k_0)(1+\gamma_0))\ln t, \]
    where
    \[ t=e^{\varkappa_p s},\qquad\gamma_0=\frac{d}{n\varkappa_p}-1. \]
    Hence
    \begin{equation} \label{eq:P_gamma_0_hat_via_unit_cube}
      \widehat P_{\gamma_0}(t)=D_{\widehat{\pmb\tau}(s)}\cB_\infty^r,
    \end{equation}
    where, as before, $r=\binom dp$.

    Thus, similar to \eqref{eq:psis_via_parallelepipeds}, we get
    \begin{equation} \label{eq:hat_psis_via_parallelepipeds}
      \bpsi_1(\widehat\La,\widehat\gT)=\liminf_{t\to+\infty}\frac{\varkappa_p\ln(\lambda_1(\widehat P_{\gamma_0}(t)))}{\ln t}\,,\qquad
      \apsi_1(\widehat\La,\widehat\gT)=\limsup_{t\to+\infty}\frac{\varkappa_p\ln(\lambda_1(\widehat P_{\gamma_0}(t)))}{\ln t}\,.
    \end{equation}

    The rest of the argument is very much the same as the corresponding part of the proof of Proposition \ref{prop:belpha_via_psis}. Let us observe that
    \[ \widehat P_\gamma(t)=t^{1-\frac{1+\gamma}{1+\gamma_0}}\widehat P_{\gamma_0}\big(t^{\frac{1+\gamma}{1+\gamma_0}}\big). \]
    This implies that
    \[ \lambda_1\big(\widehat P_\gamma(t)\big)=(t')^{1-\frac{1+\gamma_0}{1+\gamma}}\lambda_1\big(\widehat P_{\gamma_0}(t')\big) \]
    with $t'=t^{\frac{1+\gamma}{1+\gamma_0}}$. Therefore, the equality
    \[ \lambda_1\big(\widehat P_\gamma(t)\big)=1 \]
    holds if and only if
    \[ 1-\frac{1+\gamma_0}{1+\gamma}+\frac{\ln(\lambda_1(\widehat P_{\gamma_0}(t')))}{\ln t'}=0. \]
    Hence, in view of \eqref{eq:ba_via_parallelepipeds}, \eqref{eq:hat_psis_via_parallelepipeds}, we get
    \[ \gb_p=\limsup_{t\to+\infty}\left\{ (1+\gamma_0)\left( 1+\frac{\ln(\lambda_1(\widehat P_{\gamma_0}(t)))}{\ln t} \right)^{-1}-1 \right\}=
       (1+\gamma_0)\left(1+\varkappa_p^{-1}\bpsi_1(\widehat\La,\widehat\gT)\right)^{-1}-1 \]
    and
    \[ \ga_p=\liminf_{t\to+\infty}\left\{ (1+\gamma_0)\left( 1+\frac{\ln(\lambda_1(\widehat P_{\gamma_0}(t)))}{\ln t} \right)^{-1}-1 \right\}=
       (1+\gamma_0)\left(1+\varkappa_p^{-1}\apsi_1(\widehat\La,\widehat\gT)\right)^{-1}-1. \]
    Thus,
    \[ (1+\gb_p)(\varkappa_p+\bpsi_1(\widehat\La,\widehat\gT))=(1+\ga_p)(\varkappa_p+\apsi_1(\widehat\La,\widehat\gT))=d/n. \]
    It remains to apply Proposition \ref{prop:Psi_p_is_Psi_1}.
  \end{proof}

  \begin{remark} \label{rem:initial_inequalities}
    It follows from \eqref{eq:P_gamma_0_hat_via_unit_cube} that the volume of $\widehat P_{\gamma_0}(t)$ is equal to $2^r$. Hence, by Minkowski's convex body theorem, $\widehat P_{\gamma_0}(t)$ contains a non-zero point of $\widehat\La$. Thus, taking into account \eqref{eq:ba_via_parallelepipeds}, we get
    \[ \gb_p\geq\ga_p\geq\gamma_0=\frac{d}{n\varkappa_p}-1, \]
    or in terms of Schmidt--Summerer's exponents,
    \[ -\varkappa_p\leq\bPsi_p\leq\aPsi_p\leq0. \]
  \end{remark}

  \section{Transposed system} \label{sec:transposed}

  The subspace spanned by $\pmb\ell_{m+1},\ldots,\pmb\ell_d$ is the space of solutions to the system
  \[ -\tr\Theta\vec y=\vec x. \]
  As we noticed in Section \ref{sec:laurexp}, it coincides with the orthogonal complement $\cL^\bot$ for $\cL$. Denote by $\beta_p^\ast$, $\alpha_p^\ast$, $\gb_p^\ast$, $\ga_p^\ast$ the corresponding $p$-th regular and uniform Diophantine exponents of the first and of the second types for the matrix $\tr\Theta$. Obviously, they coincide with the ones corresponding to $-\tr\Theta$. The lattice constructed for $-\tr\Theta$ the very same way $\La$ was constructed for $\Theta$, would be
  \[ \begin{pmatrix}
       E_n & 0 \\
       \tr\Theta & E_m
     \end{pmatrix}\Z^d. \]
  But transposing the first $n$ and the last $m$ coordinates turns this lattice into
  \[ \begin{pmatrix}
       E_m & \tr\Theta \\
       0 & E_n
     \end{pmatrix}\Z^d=\tr T\Z^d=\La^\ast, \]
  which is the lattice, dual for $\La$. For this reason with $\tr\Theta$ we shall associate $\La^\ast$. Now, the most natural way to specify the path determining Schmidt--Summerer's exponents associated to $\tr\Theta$ is to take into account the coordinates permutation just mentioned and consider the path $\gT^\ast:s\to\pmb\tau^\ast(s)$ defined by
  \begin{equation} \label{eq:path_ast}
    \tau^\ast_1(s)=\ldots=\tau^\ast_m(s)=-ns/m,\quad\tau^\ast_{m+1}(s)=\ldots=\tau^\ast_d(s)=s.
  \end{equation}
  Denoting
  \[ \bpsi_p^\ast=\bpsi_p(\La^\ast,\gT^\ast),\quad \apsi_p^\ast=\apsi_p(\La^\ast,\gT^\ast), \]
  \[ \bPsi_p^\ast=\bPsi_p(\La^\ast,\gT^\ast),\quad \aPsi_p^\ast=\aPsi_p(\La^\ast,\gT^\ast), \]
  we see that any statement proved for an arbitrary $\Theta$ concerning the quantities $\beta_p$, $\alpha_p$, $\bpsi_p$, $\apsi_p$, $\bPsi_p$, $\aPsi_p$ remains valid if $\Theta$ is substituted by $\tr\Theta$, and the quantities $n$, $m$, $\beta_p$, $\alpha_p$, $\bpsi_p$, $\apsi_p$ are substituted by $m$, $n$, $\beta_p^\ast$, $\alpha_p^\ast$, $\bpsi_p^\ast$, $\apsi_p^\ast$, $\bPsi_p^\ast$, $\aPsi_p^\ast$, respectively. Particularly, the analogues of Propositions \ref{prop:belpha_via_psis}, \ref{prop:ba_via_Psis} hold:

  \begin{proposition} \label{prop:starred_belpha_via_starred_psis}
    We have
    \begin{equation} \label{eq:starred_belpha_via_starred_psis}
      (1+\beta_p^\ast)(1+\bpsi_p^\ast)=(1+\alpha_p^\ast)(1+\apsi_p^\ast)=d/m.
    \end{equation}
  \end{proposition}

  \begin{proposition} \label{prop:starred_ba_via_starred_Psis}
    Set $\varkappa_p^\ast=\min(p,\frac nm(d-p))$. Then
    \begin{equation} \label{eq:starred_ba_via_starred_Psis}
      (1+\gb_p^\ast)(\varkappa_p^\ast+\bPsi_p^\ast)=(1+\ga_p^\ast)(\varkappa_p^\ast+\aPsi_p^\ast)=d/m.
    \end{equation}
  \end{proposition}

  Further, same as \eqref{eq:psis_via_parallelepipeds}, we get
  \begin{equation} \label{eq:starred_psis_via_parallelepipeds}
    \bpsi_p^\ast=\liminf_{t\to+\infty}\frac{\ln(\lambda_p^\ast(P_{m/n}(t^{-n/m})))}{\ln t}\,,\qquad\apsi_p^\ast=\limsup_{t\to+\infty}\frac{\ln(\lambda_p^\ast(P_{m/n}(t^{-n/m})))}{\ln t}\,,
  \end{equation}
  where $\lambda_p^\ast$ denotes the $p$-th minimum with respect to $\La^\ast$.

  Let us show that $\bpsi_p^\ast$, $\apsi_p^\ast$ are closely connected with $\bpsi_{d-p}$, $\apsi_{d-p}$ (which, as before, are related to $\La$ and the path $\gT$ defined by \eqref{eq:path}). It follows from the definition of $P_\gamma(t)$ that there is a positive constant $c$ depending only on $\Theta$, such that
  \[ c^{-1}P_\gamma(t^{-1})\subseteq P_\gamma(t)^\ast\subseteq cP_\gamma(t^{-1}), \]
  where $P_\gamma(t)^\ast$ is the polar reciprocal body for $P_\gamma(t)$. Furthermore, it follows from Mahler's theory that
  \[ \lambda_p^\ast(P_\gamma(t)^\ast)\lambda_{d+1-p}(P_\gamma(t))\asymp1 \]
  with the implied constants depending only on $d$. Hence
  \begin{equation} \label{eq:mahler_with_no_ast}
    \lambda_p^\ast(P_\gamma(t^{-1}))\lambda_{d+1-p}(P_\gamma(t))\asymp1
  \end{equation}
  Combining \eqref{eq:starred_psis_via_parallelepipeds}, \eqref{eq:mahler_with_no_ast} and \eqref{eq:psis_via_parallelepipeds} with $p$ substituted by $d+1-p$ we get

  \begin{proposition} \label{prop:starred_psis_via_psis}
    We have
    \[ \bpsi_p^\ast=-\dfrac nm\apsi_{d+1-p}\quad\text{ and }\quad\apsi_p^\ast=-\dfrac nm\bpsi_{d+1-p}\,. \]
  \end{proposition}

  \begin{corollary} \label{cor:starred_belpha_via_psis}
    We have
    \[ (1+\beta_p^\ast)(m-n\apsi_{d+1-p})=(1+\alpha_p^\ast)(m-n\bpsi_{d+1-p})=d. \]
  \end{corollary}

  \begin{proof}
    Follows from Propositions \ref{prop:starred_belpha_via_starred_psis} and \ref{prop:starred_psis_via_psis}.
  \end{proof}

  \begin{corollary} \label{cor:starred_times_nonstarred_equals_one}
    We have
    \[ \alpha_{d+1-p}\beta_p^\ast=1\quad\text{ and }\quad\alpha_{d+1-p}^\ast\beta_p=1. \]
  \end{corollary}

  \begin{proof}
    Follows from Proposition \ref{prop:belpha_via_psis} and Corollary \ref{cor:starred_belpha_via_psis}.
  \end{proof}


  In order to obtain the corresponding relations between the exponents of the second type, let us go in the opposite direction and prove

  \begin{proposition} \label{prop:starred_equals_nonstarred}
    We have
    \[ \gb_p=\gb_{d-p}^\ast\quad\text{ and }\quad\ga_p=\ga_{d-p}^\ast. \]
  \end{proposition}

  \begin{proof}
    Let $\vec L_\sigma$, $\vec E_\sigma$, $\cJ_k$, $\cJ'_k$ be as in Section \ref{sec:laurexp}.

    We remind that the bases $\pmb\ell_1,\ldots,\pmb\ell_m,\vec e_{m+1},\ldots,\vec e_d$ and $\vec e_1,\ldots,\vec e_m,\pmb\ell_{m+1},\ldots,\pmb\ell_d$ are dual. So, if $\sigma\in\cJ_k$, $\sigma'\in\cJ'_{k'}$, then
    \[ \ast(\vec L_\sigma\wedge\vec E_{\sigma'})=\pm\vec E_{\overline\sigma}\wedge\vec L_{\overline\sigma'}, \]
    where $\ast$ denotes the Hodge star operator,
    \[ \overline\sigma=\{1,\ldots,m\}\backslash\sigma,\qquad\overline\sigma'=\{m+1,\ldots,d\}\backslash\sigma', \]
    and the sign depends on the parity of the corresponding permutation. Hence for any $\sigma\in\cJ_k$, $\sigma'\in\cJ'_{d-p-k}$, and any $\vec Z\in\wedge^p(\Z^d)$ we have
    \[ |\vec L_\sigma\wedge\vec E_{\sigma'}\wedge\vec Z|=|\vec E_{\overline\sigma}\wedge\vec L_{\overline\sigma'}\wedge\ast\vec Z|. \]
    Thus,
    \begin{equation} \label{eq:max_equals_hodged_max}
      \max_{\begin{subarray}{c} \sigma\in\cJ_k \\ \sigma'\in\cJ'_{d-p-k} \end{subarray}}
      |\vec L_\sigma\wedge\vec E_{\sigma'}\wedge\vec Z|=
      \max_{\begin{subarray}{c} \sigma'\in\cJ'_{p-m+k} \\ \sigma\in\cJ_{m-k} \end{subarray}}
      |\vec L_{\sigma'}\wedge\vec E_\sigma\wedge\ast\vec Z|,
    \end{equation}
    for each $\vec Z\in\wedge^p(\Z^d)$.

    Set $k_0^\ast=\max(0,n-(d-p))$, $k_1^\ast=\min(n,p)$. Then $k_0^\ast=k_0+p-m$, $k_1^\ast=k_1+p-m$, and the inequality $k_0\leq k\leq k_1$ is equivalent to $k_0^\ast\leq p-m+k\leq k_1^\ast$. Therefore, it follows from \eqref{eq:max_equals_hodged_max} that \eqref{eq:ba_modified} is equivalent to
    \begin{equation} \label{eq:ba_hodged}
      \max_{\begin{subarray}{c} \sigma'\in\cJ'_k \\ \sigma\in\cJ_{p-k} \end{subarray}}
      |\vec L_{\sigma'}\wedge\vec E_\sigma\wedge\ast\vec Z|\leq t^{1-(k-k_0^\ast)(1+\gamma)},\qquad k=k_0^\ast,\ldots,k_1^\ast.
    \end{equation}

    It remains to apply Proposition \ref{prop:ba_substitution} and the fact that $\ast(\wedge^p(\Z^d))=\wedge^{d-p}(\Z^d)$.
  \end{proof}

  \begin{corollary} \label{cor:starred_ba_via_Psis}
    Set $\varkappa_p^{\ast\ast}=\min(d-p,\frac mnp)=\frac mn\varkappa_p^\ast$. Then
    \[ (1+\gb_p^\ast)(\varkappa_p^{\ast\ast}+\bPsi_{d-p})=(1+\ga_p^\ast)(\varkappa_p^{\ast\ast}+\aPsi_{d-p})=d/n. \]
  \end{corollary}

  \begin{proof}
    Follows from Propositions \ref{prop:ba_via_Psis} and \ref{prop:starred_equals_nonstarred}.
  \end{proof}

  \begin{corollary} \label{cor:starred_equals_nonstarred}
    We have
    \[ \bPsi_p^\ast=\dfrac nm\bPsi_{d-p}\quad\text{ and }\quad\aPsi_p^\ast=\dfrac nm\aPsi_{d-p}\,. \]
  \end{corollary}

  \begin{proof}
    Follows from Proposition \ref{prop:starred_ba_via_starred_Psis} and Corollary \ref{cor:starred_ba_via_Psis}.
  \end{proof}

  \section{Main results in terms of Schmidt--Summerer's exponents} \label{sec:results_schmidted}

  It is interesting to rewrite \eqref{eq:dyson_transference} in terms of Schmidt--Summerer's exponents. By Propositions \ref{prop:starred_equals_nonstarred} and \ref{prop:ba_via_Psis} it becomes simply
  \begin{equation} \label{eq:Psi_very_dyson}
    \bPsi_{d-1}\leq\frac{\bPsi_1}{d-1}\,,
  \end{equation}
  which is one of the statements of Corollary \ref{cor:Psi_dyson}. But we already have an intermediate variant of this inequality! It is
  \begin{equation} \label{eq:Psi_inter_very_dyson}
    \frac{\bPsi_{p+1}}{d-p-1}\leq\frac{\bPsi_p}{d-p}\,,
  \end{equation}
  one of the statements of Corollary \ref{cor:Psi_inter_dyson}. Rewriting the corresponding statements of Corollary \ref{cor:Psi_inter_dyson} with $\La$ and $\gT$ defined by \eqref{eq:La}, \eqref{eq:path} in terms of intermediate Diophantine exponents gives Theorem \ref{t:inter_dyfel}.

  As we see, describing the splitting of Dyson's and Apfelbeck's inequalities in terms of Schmidt--Summerer's exponents given by Corollary \ref{cor:Psi_inter_dyson} is much more elegant, than in terms of Diophantine exponents. Its another attraction is its universality for all values of $n$, $m$ whose sum is equal to $d$. Moreover, Corollary \ref{cor:Psi_inter_dyson} holds actually for arbitrary lattices and paths, while Theorem \ref{t:inter_dyfel} is bound to the specific choice of those.

  Let us now translate Theorems \ref{t:inter_loranoyadenie}, \ref{t:inter_my_inequalities} into the language of Schmidt's exponents. We remind that, as we noticed in Remark \ref{rem:initial_inequalities},
  \[ -1\leq\bPsi_1\leq\aPsi_1\leq0. \]

  Theorem \ref{t:inter_loranoyadenie} turns into

  \begin{theorem} \label{t:inter_loranoyadenie_Psied}
    Suppose that the space of integer solutions of \eqref{eq:the_system} is not a one-dimensional lattice. Then
    \begin{equation} \label{eq:inter_loranoyadenie_Psied}
      \bPsi_2\leq
      \begin{cases}
        2\bPsi_1+d\cdot\dfrac{\aPsi_1-\bPsi_1}{n+n\aPsi_1}\,,\quad\text{ if }\ \aPsi_1\neq-1, \\
        2\bPsi_1+d\cdot\dfrac{\aPsi_1-\bPsi_1 \vphantom{\frac{\big|}{}} }{m-n\aPsi_1}\,.
      \end{cases}
    \end{equation}
  \end{theorem}

  Theorem \ref{t:inter_my_inequalities} turns into

  \begin{theorem} \label{t:inter_my_inequalities_Psied}
    We have
    \begin{equation} \label{eq:inter_my_inequalities_Psied}
      \aPsi_2\leq
      \begin{cases}
        \dfrac{(d-2)\aPsi_1}{(n-1)+n\aPsi_1}\,,\quad\text{ if }\ \aPsi_1\geq\dfrac{m-n}{2n}\,, \\
        \dfrac{(d-2)\aPsi_1 \vphantom{\frac{\big|}{}} }{(m-1)-n\aPsi_1}\,,\quad\text{ if }\ \aPsi_1\leq\dfrac{m-n}{2n}\,.
      \end{cases}
    \end{equation}
  \end{theorem}

  As we see, this point of view relieves us of singling out the case $m=1$. In the next Section we prove Theorems \ref{t:inter_loranoyadenie_Psied}, \ref{t:inter_my_inequalities_Psied}.

  \section{Proof of Theorems \ref{t:inter_loranoyadenie_Psied}, \ref{t:inter_my_inequalities_Psied}} \label{sec:proofs}

  Let $\La$ and $\gT$ be fixed by \eqref{eq:La} and \eqref{eq:path}. The following observation is the crucial point for proving Theorems \ref{t:inter_loranoyadenie_Psied}, \ref{t:inter_my_inequalities_Psied}.

  \begin{lemma} \label{l:main}
    Suppose $s,s'\in\R_+$ satisfy the conditions
    \begin{equation} \label{eq:main_subseteq}
      \lambda_1(\cB(s))\cB(s)\subseteq\lambda_1(\cB(s'))\cB(s'),
    \end{equation}
    \begin{equation} \label{eq:main_equals}
      \lambda_1(\cB(s'))=\lambda_2(\cB(s')).
    \end{equation}
    Then
    \begin{equation} \label{eq:main}
      \psi_2(s)\leq
      \begin{cases}
        \psi_1(s)+d\cdot\dfrac{\psi_1(s')-\psi_1(s)}{n+n\psi_1(s')}\,,\quad\text{ if }\ s'\leq s\ \text{ and }\ \psi_1(s')\neq-1, \\
        \psi_1(s)+d\cdot\dfrac{\psi_1(s')-\psi_1(s) \vphantom{\frac{\big|}{}} }{m-n\psi_1(s')}\,,\quad\text{ if }\ s'\geq s.
      \end{cases}
    \end{equation}
  \end{lemma}

  \begin{proof}
    Suppose that $s'\leq s$. Then it follows from \eqref{eq:main_subseteq} and \eqref{eq:main_equals} that
    \[ \lambda_1(\cB(s))e^s=\lambda_1(\cB(s'))e^{s'}\geq1 \]
    and
    \[ \lambda_2(\cB(s))e^{-ms/n}\leq\lambda_2(\cB(s'))e^{-ms'/n}=\lambda_1(\cB(s'))e^{-ms'/n}, \]
    i.e.
    \begin{equation} \label{eq:s_prime_small_via_s}
      s(1+\psi_1(s))=s'(1+\psi_1(s'))\geq0
    \end{equation}
    and
    \begin{equation} \label{eq:s_prime_small_ps_2_leq_psi_1}
      s(\psi_2(s)-m/n)\leq s'(\psi_1(s')-m/n)
    \end{equation}
    Combining \eqref{eq:s_prime_small_via_s} and \eqref{eq:s_prime_small_ps_2_leq_psi_1} we get the first inequality of \eqref{eq:main}.

    Suppose now that $s'\geq s$. Then it follows from \eqref{eq:main_subseteq} and \eqref{eq:main_equals} that
    \[ \lambda_1(\cB(s))e^{-ms/n}=\lambda_1(\cB(s'))e^{-ms'/n}<1 \]
    and
    \[ \lambda_2(\cB(s))e^s\leq\lambda_2(\cB(s'))e^{s'}=\lambda_1(\cB(s'))e^{s'}, \]
    i.e.
    \begin{equation} \label{eq:s_prime_large_via_s}
      s(\psi_1(s)-m/n)=s'(\psi_1(s')-m/n)<0
    \end{equation}
    and
    \begin{equation} \label{eq:s_prime_large_ps_2_leq_psi_1}
      s(1+\psi_2(s))\leq s'(1+\psi_1(s'))
    \end{equation}
    Combining \eqref{eq:s_prime_large_via_s} and \eqref{eq:s_prime_large_ps_2_leq_psi_1} we get the second inequality of \eqref{eq:main}.
  \end{proof}

  For each $\vec z=\tr{(z_1,\ldots,z_d)}\in\R^d$ and each $s>0$ let us set
  \[ \mu_s(\vec z)=e^{-s}\max_{1\leq i\leq m}|z_i|\qquad\text{ and }\qquad\nu_s(\vec z)=e^{ms/n}\max_{m<i\leq d}|z_i|. \]
  Then
  \[ \cB(s)=\Big\{ \vec z\in\R^d \,\Big|\, \mu_s(\vec z)\leq1,\ \nu_s(\vec z)\leq1 \Big\}. \]
  The parallelepiped $\lambda_1(\cB(s))\cB(s)$ contains no non-zero points of $\La$ in its interior and contains at least one pair of such points in its boundary. Of these points let us choose an arbitrary point and denote it by $\vec v_s$. Obviously, the maximal of the quantities $\mu_s(\vec v_s)$, $\nu_s(\vec v_s)$ equals $\lambda_1(\cB(s))$.

  \begin{corollary} \label{cor:for_inter_loranoyadenie}
    For each $s>0$, such that
    \begin{equation} \label{eq:for_inter_loranoyadenie_condition}
      \mu_s(\vec v_s)=\nu_s(\vec v_s)=\lambda_1(\cB(s)),
    \end{equation}
    there are $s',s''>0$, such that
    \[ s(1+\psi_1(s))\leq s'\leq s\leq s''\leq s(1-(n/m)\psi_1(s)) \]
    and
    \begin{equation} \label{eq:for_inter_loranoyadenie}
      \Psi_2(s)\leq
      \begin{cases}
        2\psi_1(s)+d\cdot\dfrac{\psi_1(s')-\psi_1(s)}{n+n\psi_1(s')}\,,\quad\text{ if }\ \psi_1(s')\neq-1, \\
        2\psi_1(s)+d\cdot\dfrac{\psi_1(s'')-\psi_1(s) \vphantom{\frac{\big|}{}} }{m-n\psi_1(s'')}\,.
      \end{cases}
    \end{equation}
  \end{corollary}

  \begin{proof}
    Let us show that the relation $\mu_s(\vec v_s)=\lambda_1(\cB(s))$ implies the existence of an $s'\leq s$ satisfying the conditions of Lemma \ref{l:main}. Denote $\lambda=\lambda_1(\cB(s))$. Let
    \[ \cP_\nu=\Big\{ \vec z\in\R^d \,\Big|\, \mu_s(\vec z)\leq\lambda,\ \nu_s(\vec z)\leq\nu\lambda \Big\} \]
    be the minimal (w.r.t inclusion) parallelepiped containing no non-zero points of $\La$ in its interior. The existence of such a parallelepiped follows from Minkowski's convex body theorem. It also implies that $1\leq\nu\leq\lambda^{-d/n}$. Then
    \[ \lambda\cB(s)\subseteq\cP_{\nu}=\lambda'\cB(s'), \]
    where $\lambda'=\lambda\nu^{n/d}$, $s'=s-(n/d)\ln\nu$. For $\lambda'$, $s'$ we have
    \[ \lambda'\geq\lambda,\qquad s+\ln\lambda\leq s'\leq s. \]
    On the other hand, $\cP_\nu$ contains non-collinear points of $\La$ in its boundary, so $\lambda_1(\cB(s'))=\lambda_2(\cB(s'))=\lambda'$. Thus, $s$, $s'$ satisfy \eqref{eq:main_subseteq}, \eqref{eq:main_equals}.

    Now let us consider the relation $\nu_s(\vec v_s)=\lambda_1(\cB(s))$. By Minkowski's convex body theorem there is a $\mu$ in the interval $1\leq\mu\leq\lambda^{-d/m}$, such that the parallelepiped
    \[ \cQ_\mu=\Big\{ \vec z\in\R^d \,\Big|\, \mu_s(\vec z)\leq\mu\lambda,\ \nu_s(\vec z)\leq\lambda \Big\} \]
    contains no non-zero points of $\La$ in its interior, but contains non-collinear points of $\La$ in its boundary. Then
    \[ \lambda\cB(s)\subseteq\cQ_{\mu}=\lambda''\cB(s''), \]
    where $\lambda''=\lambda\mu^{m/d}$, $s''=s+(n/d)\ln\mu$. For $\lambda''$, $s''$ we have
    \[ \lambda''\geq\lambda,\qquad s\leq s''\leq s-(n/m)\ln\lambda. \]
    Besides that, $s$, $s''$ also satisfy \eqref{eq:main_subseteq}, \eqref{eq:main_equals}, since $\lambda_1(\cB(s''))=\lambda_2(\cB(s''))=\lambda''$.

    It remains to apply Lemma \ref{l:main}.
  \end{proof}

  Having Corollary \ref{cor:for_inter_loranoyadenie}, it is easy now to prove Theorem \ref{t:inter_loranoyadenie_Psied}.

  First, let us notice that if the system \eqref{eq:the_system} has a non-zero integer solution, then it has two linearly independent integer solutions, so in this case $\aPsi_1=\bPsi_1=-1$, $\bPsi_2=-2$, which implies \eqref{eq:inter_loranoyadenie_Psied}.

  Next, let us suppose that the system \eqref{eq:the_system} has no non-zero integer solutions. Then there are infinitely many local minima of $\psi_1(s)$, each of them satisfies \eqref{eq:for_inter_loranoyadenie_condition}, and the sequence of these local minima tends to $\infty$. Moreover, $s'$ and $s''$ from Corollary \ref{cor:for_inter_loranoyadenie} tend to $\infty$ as $s$ tends to $\infty$. Indeed, since \eqref{eq:the_system} has no non-zero integer solutions, we have
  \[ e^{s(1+\psi_1(s))}=e^s\lambda_1(\cB(s))=\lambda_1(e^{-s}\cB(s))\to\infty\ \ \text{ as }\ \ s\to\infty, \]
  so
  \begin{equation} \label{eq:tending_to_infty}
    s(1+\psi_1(s))\to\infty\ \ \text{ as }\ \ s\to\infty.
  \end{equation}
  Particularly, it follows from \eqref{eq:tending_to_infty} that $\psi_1(s)$ is eventually greater than $-1$ (it can actually be shown that $\psi_1(s)>-1$ starting with the second local minimum of $\psi_1(s)$).
  Therefore,
  \begin{equation} \label{eq:for_inter_loranoyadenie_liminfsup}
    \bPsi_2\leq\liminf\Psi_2(s)\leq
    \begin{cases}
      2\liminf\psi_1(s)+d\cdot\limsup\dfrac{\psi_1(s')-\psi_1(s)}{n+n\psi_1(s')}\,, \\
      2\liminf\psi_1(s)+d\cdot\limsup\dfrac{\psi_1(s'')-\psi_1(s) \vphantom{\frac{\big|}{}} }{m-n\psi_1(s'')}\,,
    \end{cases}
  \end{equation}
  where the $\liminf$ and the $\limsup$ are taken over the set of local minima of $\psi_1(s)$. Since $\psi_1(s)$ is never positive, both denominators in \eqref{eq:for_inter_loranoyadenie_liminfsup} are eventually positive. Therefore, \eqref{eq:for_inter_loranoyadenie_liminfsup} implies \eqref{eq:inter_loranoyadenie_Psied}.

  \begin{corollary} \label{cor:for_inter_my_inequalities}
    Suppose that the system \eqref{eq:the_system} has no non-zero integer solutions. Then for each $s>0$ there is an $s'>0$, such that $s(1+\psi_1(s))\leq s'\leq s(1-(n/m)\psi_1(s))$, and
    \begin{equation} \label{eq:for_inter_my_inequalities}
      \Psi_2(s)\leq
      \begin{cases}
        \dfrac{(d-2)\psi_1(s')}{(n-1)+n\psi_1(s')}\,,\quad\text{ if }\ \psi_1(s')\geq\dfrac{m-n}{2n}\,, \\
        \dfrac{(d-2)\psi_1(s') \vphantom{\frac{\big|}{}} }{(m-1)-n\psi_1(s')}\,,\quad\text{ if }\ \psi_1(s')\leq\dfrac{m-n}{2n}\,.
      \end{cases}
    \end{equation}
  \end{corollary}

  \begin{proof}
    Assume that $\mu_s(\vec v_s)=\lambda_1(\cB(s))$. Then the same argument as in the proof of Corollary \ref{cor:for_inter_loranoyadenie} shows that there is an $s'$, such that $s(1+\psi_1(s))\leq s'\leq s$, and
    \begin{equation} \label{eq:corollary_n}
      \Psi_2(s)\leq2\psi_1(s)+d\cdot\dfrac{\psi_1(s')-\psi_1(s)}{n+n\psi_1(s')}\,,
    \end{equation}
    unless $\psi_1(s')=-1$.
    By Corollary \ref{cor:precise_inter_dyfel} we have
    \begin{equation} \label{eq:psi_sandwiched_by_Psi}
      \frac{d-1}{d-2}\Psi_2(s)\leq\psi_1(s)\leq\frac 12\Psi_2(s).
    \end{equation}
    If $\psi_1(s')=-1$, then \eqref{eq:psi_sandwiched_by_Psi} implies \eqref{eq:for_inter_my_inequalities}. Suppose that $\psi_1(s')\neq-1$. Then, taking into account that
    \begin{equation*} 
      2-\frac{d}{n+n\psi_1(s')}\geq0\quad\text{ if and only if }\quad\psi_1(s')\geq\frac{m-n}{2n}\,,
    \end{equation*}
    we conclude from \eqref{eq:corollary_n} and \eqref{eq:psi_sandwiched_by_Psi} that
    \begin{equation} \label{eq:if_s_prime_is_less_than_s}
      \Psi_2(s)\leq
      \begin{cases}
        \qquad\ 2\psi_1(s')\,,\qquad\quad\,\ \text{ if }\ \psi_1(s')\geq\dfrac{m-n}{2n}\,, \\
        \dfrac{(d-2)\psi_1(s')}{(m-1)-n\psi_1(s')}\,,\quad\text{ if }\ \psi_1(s')\leq\dfrac{m-n \vphantom{\frac{\big|}{}} }{2n}\,.
      \end{cases}
    \end{equation}

    Assume now that $\nu_s(\vec v_s)=\lambda_1(\cB(s))$. Then the same argument as in the proof of Corollary \ref{cor:for_inter_loranoyadenie} shows that there is an $s''$, such that $s\leq s''\leq s(1-(n/m)\psi_1(s))$, and
    \begin{equation} \label{eq:corollary_m}
      \Psi_2(s)\leq2\psi_1(s)+d\cdot\dfrac{\psi_1(s'')-\psi_1(s)}{m-n\psi_1(s'')}\,.
    \end{equation}
    Taking into account that
    \begin{equation*} 
      2-\frac{d}{m-n\psi_1(s'')}\geq0\quad\text{ if and only if }\quad\psi_1(s'')\leq\frac{m-n}{2n}\,,
    \end{equation*}
    we conclude from \eqref{eq:corollary_m} and \eqref{eq:psi_sandwiched_by_Psi} that
    \begin{equation} \label{eq:if_s_prime_is_greater_than_s}
      \Psi_2(s)\leq
      \begin{cases}
        \dfrac{(d-2)\psi_1(s'')}{(n-1)+n\psi_1(s'')}\,,\quad\text{ if }\ \psi_1(s'')\geq\dfrac{m-n}{2n}\,, \\
        \qquad\ 2\psi_1(s'')\,,\qquad\quad\ \text{ if }\ \psi_1(s'')\leq\dfrac{m-n \vphantom{\frac{\big|}{}} }{2n}\,.
      \end{cases}
    \end{equation}

    Since $\psi_1(s')$ and $\psi_1(s'')$ are negative, we have
    \[ 2\psi_1(s')\leq\dfrac{(d-2)\psi_1(s')}{(n-1)+n\psi_1(s')}\,,\qquad\text{ if }\ \psi_1(s')\geq\dfrac{m-n}{2n}\,, \]
    and
    \[ 2\psi_1(s'')\leq\dfrac{(d-2)\psi_1(s'')}{(m-1)-n\psi_1(s'')}\,,\quad\text{ if }\ \psi_1(s'')\leq\dfrac{m-n}{2n}\,. \]
    Therefore, \eqref{eq:if_s_prime_is_less_than_s} and \eqref{eq:if_s_prime_is_greater_than_s} imply the desired statement.
  \end{proof}

  Deriving Theorem \ref{t:inter_my_inequalities_Psied} from Corollary \ref{cor:for_inter_my_inequalities} is even easier than deriving Theorem \ref{t:inter_loranoyadenie_Psied} from Corollary \ref{cor:for_inter_loranoyadenie}.

  If the system \eqref{eq:the_system} has a non-zero integer solution, then $\aPsi_1=-1<\frac{m-n}{2n}$\,, and \eqref{eq:inter_my_inequalities_Psied} follows from \eqref{eq:Psi_inter_dyson}. Suppose now that \eqref{eq:the_system} has no non-zero integer solutions. Then it follows from \eqref{eq:tending_to_infty} that $s'$ from Corollary \ref{cor:for_inter_my_inequalities} tends to $\infty$ as $s$ tends to $\infty$. Hence, taking $\limsup$ of both sides in \eqref{eq:for_inter_my_inequalities}, we get \eqref{eq:inter_my_inequalities_Psied}.

  \section*{Acknowledgements}

  The author would like to thank the Department of Mathematics of the University of York for the warm welcome during his research visit in February -- March 2011, for this was the time when this paper was conceived.

\vskip 10mm

\noindent
Oleg N. {\sc German} \\
Moscow Lomonosov State University \\
Vorobiovy Gory, GSP--1 \\
119991 Moscow, RUSSIA \\
\emph{E-mail}: {\fontfamily{cmtt}\selectfont german@mech.math.msu.su, german.oleg@gmail.com}

\end{document}